\documentclass[12pt]{amsart}
\usepackage{amssymb,latexsym}
\usepackage{enumerate}

\makeatletter
\@namedef{subjclassname@2010}{%
  \textup{2010} Mathematics Subject Classification}
\makeatother

\newtheorem{thm}{Theorem}[section]
\newtheorem{cor}[thm]{Corollary}
\newtheorem{lemma}[thm]{Lemma}
\newtheorem{prop}[thm]{Proposition}

\theoremstyle{definition}
\newtheorem{defin}[thm]{Definition}

\newtheorem{exa}[thm]{Example}

\numberwithin{equation}{section}

\begin{document}

\baselineskip=17pt

\title[Bergman-Lorentz spaces]{Bergman-Lorentz spaces on tube domains over symmetric cones}
\author{DAVID BEKOLLE, JOCELYN GONESSA AND CYRILLE NANA}
%\address{Current address of Bekolle: University of Yaound\'e 1, Faculty of Sciences, Department of Mathematics, P. O. Box 812, Yaound\'e-Cameroon.}
\address{University of Ngaound\'er\'e, Faculty of Science, Department of Mathematics, P. O. Box 454, Ngaound\'er\'e, Cameroon}
\email{dbekolle@univ-ndere.cm}
\address{Universit\'e de Bangui, Facult\'e des Sciences, D\'epartement de math\'ematiques et Informatique, BP. 908 Bangui, R\'epublique Centrafricaine}
\email{gonessa.jocelyn@gmail.com}
\date{}
\address{Faculty of Science, Department of Mathematics, University of Buea, P.O. Box 63, Buea, Cameroon}
\email{{\tt nana.cyrille@ubuea.cm}}

\maketitle
  
%%%%%%%%%%%%%%%beginig  %%%%%%%%%%%%%%%%%%%%%%%%%%%%%%%%%%%%
\begin{abstract}
We study Bergman-Lorentz spaces on tube domains over symmetric cones, i.e. spaces of holomorphic functions which belong to Lorentz spaces $L(p, q).$ We establish boundedness and surjectivity of Bergman projectors from Lorentz spaces to the corresponding Bergman-Lorentz spaces and real interpolation between  Bergman-Lorentz spaces. Finally we ask a question whose positive answer would enlarge the interval of parameters $p\in (1, \infty)$ such that the relevant Bergman projector is bounded on $L^p$ for cones of rank $r\geq 3.$
\end{abstract}

\section{Introduction} 
The notations and definitions are those of \cite{FK}. We denote by $\Omega$ an irreducible symmetric cone in $\mathbb R^n$ with rank $r$ and determinant $\Delta.$ We denote $T_\Omega = \mathbb R^n +i\Omega$ the tube domain in $\mathbb C^n$ over $\Omega.$ For $\nu \in \mathbb R,$ we define the weighted measure $\mu$ on $T_\Omega$ by $d\mu (x+iy) = \Delta^{\nu-\frac nr} (y)dxdy.$ We consider Lebesgue spaces $L_\nu^p$ and Lorentz spaces $L_\nu (p, q)$ on the measure space $(T_\Omega, d\mu).$ The Bergman space $A^p_\nu$ (resp. the Bergman-Lorentz space $A_\nu (p, q))$ is the subspace of $L_\nu^p$ (resp. of $L_\nu (p, q))$ consisting of holomorphic functions. \\
Our first result is the following.

\begin{thm}\label{th1}
Let $1< p\leq \infty$ and $1\leq q \leq \infty.$ 
\begin{enumerate}
\item 
For $\nu \leq \frac nr -1,$ the Bergman-Lorentz space $A_\nu (p, q)$ is trivial, i.e $A_\nu (p, q)=\{0\}.$
\item
Suppose $\nu >\frac nr -1.$ Equipped with the norm induced by the Lorentz space $L_\nu (p, q),$ the Bergman-Lorentz space $A_\nu (p, q)$ is a Banach space.
\end{enumerate}
\end{thm}

For $p=2,$ the Bergman space $A_\nu^2$ is a closed subspace of the Hilbert space $L_\nu^2$ and the Bergman projector $P_\nu$ is the orthogonal projector from $L_\nu^2$ to $A_\nu^2.$ We adopt the notation 
$$Q_\nu = 1+\frac \nu{\frac nr -1}.$$
 Our boundedness theorem for Bergman projectors on Lorentz spaces is the following.

\begin{thm}\label{th2}
Let $\nu >\frac nr -1$ and $1\leq q\leq \infty.$ The weighted Bergman projector $P_\gamma, \hskip 2truemm \gamma \geq \nu + \frac \nu r -1$ (resp. the Bergman projector $P_\nu)$ extends to a bounded  operator from $L_\nu (p, q)$ to $A_\nu (p, q)$ for all $1<  p < Q_\nu$ (resp. for all $1+Q_\nu^{-1}<p<1+Q_\nu$). In this case, under the restriction  $1<q<\infty$ (resp. $1\leq q<\infty)$, then  $P_\gamma$ for $\gamma$ sufficiently large (resp.  $P_\nu)$ is the identity on $A_\nu (p, q).$
\end{thm}

For the Bergman projector $P_\nu,$ following recent developments, this theorem is extended below to a larger interval of exponents $p$ on tube domains over Lorentz cones ($r=2)$ (see section 6). Finally our real interpolation theorem between Bergman-Lorentz spaces is the following.

\begin{thm}\label{th3}
Let $\nu >\frac nr -1.$ 
%and $1\leq q<\infty.$ 
\begin{enumerate}
%\item
%For all $1\leq p_0 < p_1 < 1+Q_\nu,$ the real interpolation space $[A_\nu^{p_0}, A_\nu^{p_1}]_{\theta, q}$ identifies with $A_\nu (p, q)$ with equivalence of norms.
\item
For all $1< p_0 < p_1 < Q_\nu$  (resp. $1+Q_\nu^{-1} <  p_0 <p_1 < 1+ Q_\nu),\hskip 2truemm 1\leq q_0, q_1\leq \infty$  and  $0<\theta<1,$ the real interpolation space 
$$[A_\nu (p_0, q_0), A_\nu (p_1, q_1)]_{\theta, q}$$ 
identifies with $A_\nu (p, q), \quad \frac 1p = \frac {1-\theta}{p_0}+\frac \theta {p_1}, \hskip2 truemm 1< q<\infty$ (resp. $1\leq q<\infty$) with equivalence of norms.
\item 
For  $1< p_0 < 1+Q_\nu^{-1}, \hskip 1truemm Q_\nu < p_1 < 1+ Q_\nu, \hskip 1truemm 1\leq q_0, q_1\leq \infty,$  the Bergman-Lorentz spaces $A_\nu (p, q), \hskip 2truemm 1+Q_\nu^{-1} < p < Q_\nu, \hskip 1truemm 1\leq q<\infty$ are real interpolation spaces between $A_\nu (p_0, q_0)$ and $A_\nu (p_1, q_1)$ with equivalence of norms.
\end{enumerate} 
\end{thm}

These results were first presented in the PhD dissertation of the second author \cite{G}.\\
\indent 
The plan of the paper is as follows. In section 2, we overview definitions and properties of Lorentz spaces on a non-atomic $\sigma$-finite measure space. This section encloses  results on real interpolation between Lorentz spaces. In section 3, we define Bergman-Lorentz spaces on a tube domain $T_\Omega$ over a symmetric cone $\Omega$ and we establish Theorem \ref{th1}. In section 4, we study the density of the subspace $A_\nu (p, q)\cap A_\gamma^t$ in the Banach space $A_\nu (p, q)$ for $\nu, \gamma >\frac nr -1, \hskip 1truemm 1< p, t \leq \infty,  \hskip 1truemm 1\leq q < \infty.$  Relying on boundedness results for the Bergman projectors $P_\gamma, \hskip 2truemm \gamma \geq  \nu+\frac nr -1$ and $P_\nu$ on Lebesgue spaces $L^p_\nu$ \cite{BBGNPR, BBGRS}, we then provide a proof of Thorem \ref{th2}. In section 5, we prove Theorem \ref{th3}. In section 6, we ask a question whose positive answer would enlarge the interval of paramaers $p\in (1, \infty)$ such that the Bergman projector $P_\nu$ is bounded on $L^p_\nu$ for upper rank cones ($r\geq 3$). A second question addresses the density of the subspace $A_\nu (p, q)\cap A_\gamma^t$ in the Banach space $A_\nu (p, q).$  The third question concerns a possible extension of Theorem 1.3.
 
\section{Lorentz spaces on measure spaces}
\subsection{Definitions and preliminary topological properties}
Throughout this section, the notation $(E, \mu)$ is fixed for a non-atomic $\sigma$-finite measure space. We refer to \cite{SW}, \cite{H}, \cite{BS} and \cite{BL}. Also cf. \cite{G}.
\begin{defin}
Let $f$ be a measurable function on $(E , \mu).$ The distribution function $\mu_f$ of $f$ is defined on $[0, \infty)$ by
$$\mu_f (\lambda) = \mu (\{x\in E: |f(x)|>\lambda\}).$$
The non-increasing rearrangement function $f^\star$ of $f$ is defined on $[0, \infty)$ by
$$f^\star (t) = \inf \{\lambda \geq 0: \mu_f (\lambda) \leq t\}.$$
\end{defin}

%\begin{lemma}\label{2.2}
%Let $f, g$ be two measurable functions on $(E , \mu).$ Then
%$$(f+g)^\star (t_1 + t_2) \leq f^\star (t_1)+g^\star (t_2)$$
%for all $t_1, t_2 \geq 0.$
%\end{lemma}

\begin{thm}[Hardy-Littlewood Theorem]\label{HL}
Let $f$ and $g$ be two measurable functions on $(E , \mu).$ Then 
$$\int_{E} |f(x)g(x)|d\mu (x) \leq \int_0^\infty f^\star (s)g^\star (s)ds.$$
In particular, let $g$ be a positive measurable function on $(E , \mu)$ and let $F$ be a measurable subset of $E$ of bounded measure $\mu (F).$ Then 
$$\int_F g(x)d\mu(x) \leq \int_0^{\mu (F)} g^\star (s)ds.$$
\end{thm}

\begin{defin}
Let $1\leq p, q \leq \infty.$ The Lorentz space $L (p, q)$ is the space of measurable functions on $(E , \mu)$ such that
$$||f||^\star_{p, q} = \left (\int_0^\infty \left (t^{\frac 1p} f^\star (t)\right )^q \frac {dt}t\right ) ^{\frac 1q} < \infty \quad \rm {if} \quad 1\leq p < \infty \quad \rm {and} \quad 1\leq q < \infty$$
(resp.
$$||f||^\star_{p, \infty} = \sup \limits_{t>0} \hskip 1truemm t^{\frac 1p} f^\star (t) < \infty \quad {\rm {if}} \quad 1\leq p \leq \infty).$$
%be a measurable function on $(E , \mu).$ The distribution function $\mu_f$ of $f$ is defined on $[0, \infty)$ by 
%$$\mu_f (\lambda) = \mu (\{z\in T_\Omega: |f(z)|>\lambda\}).$$
%The non-increasing rearrangement function $f^\star$ of $f$ is defined on $[0, \infty)$ by
%$$f^\star (t) = \inf \{\lambda \geq 0: \mu_f (\lambda) \leq t\}.$$
\end{defin}
We recall that for $p=\infty$ and $q< \infty,$ this definition gives way to the space of vanishing almost everywhere functions on $(E, \mu).$ It is also well known that for $p=q,$ the Lorentz space $L (p, p)$ coincides with the Lebesgue space $L^p (E, d\mu).$ More precisely, we have the equality
\begin{equation}\label{Lp}
||f||_p = \left (\int_0^\infty f^\star (t)^p dt\right )^{\frac 1p}.
\end{equation}
In the sequel we shall adopt the following notation:
$$L^p = L^p (E, d\mu).$$

We shall need the following two results.
\begin{prop}
The functional $||\cdot||^\star_{p, q}$ is a quasi-norm {\rm {(it satisfies all properties of a norm except the triangle inequality)}} on $L (p, q)$. This functional is a norm if $1\leq q \leq p < \infty$ or $p=q=\infty.$ 
%Moreover, for all $1\leq p, q < \infty,$ we have
%$$|\int_{E} f(x)g(x)d\mu (x)| \leq ||f||^\star_{p, q} ||g||^\star_{p', q'} \quad (f\in L (p, q), \hskip 2truemm g\in L (p', q')).$$
%Here $p'$ and $q'$ are the respective conjugate exponents of $p$ and $q.$
\end{prop}

\begin{lemma}
Let $1\leq p < \infty.$ Then for every measurable function $f$ on $(E , \mu),$ we have
$$||f||_{p, \infty}^\star = \sup \limits_{\lambda >0} \lambda  \mu_f (\lambda)^{\frac 1p}.$$
\end{lemma}

%\begin{lemma}\label{2.7}
%Let $g$ be a positive measurable function on $(E , \mu)$ and let $F$ be a measurable subset of $E$ of bounded measure $\mu (F).$ Then 
%$$\int_F g(x)d\mu(x) \leq \int_0^{\mu (F)} g^\star (s)ds.$$
%Moreover, for every positive number $t,$ there exists a measurable subset $E_t$ of $E$ satisfying $\mu (E_t) = t$ such that
%$$\int_{E_t} g(x)d\mu(x) = \int_0^{t} g^\star (s)ds.$$
%\end{lemma}

We next define a norm on $L (p, q)$ which is equivalent to the quasi-norm $||\cdot||^\star_{p, q}.$ For a measurable function $f$ on $(E , \mu),$ we define the function $f^{\star \star}$ on $[0, \infty)$ by 
$$f^{\star \star} (t) = \frac 1t \int_0^t f^\star (u)du.$$
For $1\leq p, q \leq \infty,$ we define the functional $||\cdot||_{p, q}$ by
$$||f||_{p, q}= \left (\int_0^\infty \left (t^{\frac 1p} f^{\star \star} (t)\right )^q \frac {dt}t\right )^{\frac 1q}  \quad \rm {if} \quad 1\leq p < \infty \quad \rm {and} \quad 1\leq q < \infty$$
(resp.
$$||f||_{p, \infty} = \sup \limits_{t>0} \hskip 1truemm t^{\frac 1p} f^{\star \star} (t)  \quad \rm {if} \quad 1\leq p \leq \infty).$$

\begin{prop}
Let $1< p \leq \infty$ and $1\leq q \leq \infty.$ The functional $||\cdot||_{p, q}$ is a norm on $L (p, q).$ 
\end{prop}

\begin{lemma}
Let $1< p \leq \infty$ and $1\leq q \leq r \leq \infty.$ Then for every measurable function $f$ on $(E , \mu),$ we have
$$||f||_{p, r} \leq ||f||_{p, q}$$
and
$$||f||^\star_{p, r} \leq ||f||^\star_{p, q}.$$
\end{lemma}

\noindent
The following lemma asserts that the norm $||\cdot||_{p, q}$ is equivalent to the quasi-norm $||\cdot||^\star_{p, q}$ when $1< p \leq \infty$ and $1\leq q \leq \infty.$ 

\begin{lemma}
For all $1< p< \infty,  \hskip 2truemm 1\leq q \leq \infty$ and for every $f\in L_\nu (p, q),$ we have
$$||f||^\star_{p, q} \leq ||f||_{p, q} \leq \frac p{p-1}||f||^\star_{p, q}.$$
\end{lemma}

We now state the following theorem.

\begin{thm}
Let $1< p < \infty, \hskip 1truemm 1\leq q \leq \infty.$  Equipped with the norm $||\cdot||_{p, q},$ the Lorentz space $L (p, q)$ is a Banach space.
\end{thm}

We shall also need the following proposition.

\begin{prop}\label{(2.12)}
Let $1< p < \infty$ and $1\leq q \leq \infty.$ Every Cauchy sequence in the Lorentz space $(L (p, q), ||\cdot||_{p, q})$ contains a subsequence which converges a.e. to its limit in $L (p, q).$ 
\end{prop}

%\subsection{Density in  Lorentz spaces.}  We recall the following density theorem.

%\begin{thm}\cite[Theorem 2.9]{BS}
%Suppose that 

%\end{thm}

%By simple function, we mean a function which can be written in the form
%\begin{equation*}
%f(x) = \sum_{j=1}^N c_j \chi_{E_j} (x),
%\end{equation*}
%where $c_1,\cdots,c_N$ are complex numbers, $E_1,\cdots,E_N$ are pairwise disjoint sets of finite measure and $\chi_E$ denotes the characteristic function of the set $E.$\\ 
%We recall the following theorem (cf. e.g. \cite{H}).

%\begin{thm}
%Let $1\leq p <  \infty$ and $1\leq q < \infty.$ The subspace of simple functions is dense in the Banach space $(L (p, q), ||\cdot||_{p, q}).$

%\end{thm}

We finally record the following duality theorem \cite{H}.

\begin{thm}\label{duality}
$(1)$ \hskip 2truemm
Let $1\leq p <\infty.$ The topological dual space $(L (p, 1))'$ of the Lorentz space $L (p, 1)$ identifies with the Lorentz space $L (p', \infty)$ with respect to the duality pairing
$$(\star) \quad (f, g) = \int_{T_\Omega} f(z)\overline{g(z)}d\mu (z).$$
$(2)$ \hskip 2truemm
Let $1< p < \infty$  and $1< q < \infty.$ The topological dual space $(L (p, q))'$ of the Lorentz space $L (p, q)$ identifies with the Lorentz space $L (p', q')$ with respect to the same duality pairing $(\star).$
\end{thm}

\subsection{Interpolation via the real method between Lorentz spaces}
We begin with an overview of the theory of real interpolation between Banach spaces (cf. e.g. \cite{BS}).
\begin{defin} 
A pair $(X_0, X_1)$ of Banach  spaces is called a compatible couple if there is some Hausdorff topological vector space in which $X_0$ and $X_1$ are continuously embedded.
\end{defin}

\begin{defin}
Let $(X_0, X_1)$ be a compatible couple of Banach spaces. Denote $\overline {X}=X_0 + X_1.$ Let $t>0$ and $a\in \overline {X}.$ We define the functional  $K(t, a, \overline {X})$ by
$$K(t, a, \overline {X}) = \inf \hskip 2truemm \{||a_0||_{X_0} + t||a_1||_{X_1}: a=a_0 +a_1, a_0 \in X_0, a_1\in X_1\}.$$
For $0<\theta <1, \hskip 1truemm 1\leq q <\infty $ or $0\leq \theta \leq 1, \hskip 1truemm q=\infty,$ the real interpolation space between $X_0$ and $X_1$ is the space
$$[X_0, X_1]_{\theta, q} := \{a\in \overline {X}: ||a||_{\theta, q, \overline {X}}  < \infty\}$$
with  
$$||a||_{\theta, q, \overline {X}} := \left (\int_0^\infty \left (t^{-\theta} K(t, a, \overline {X})\right )^q \frac {dt}t\right )^{\frac 1q}$$
if $1\leq q< \infty$ (resp.
$$||a||_{\theta, \infty, \overline {X}} := \sup \limits_{t>0} t^{-\theta} K(t, a, \overline {X})$$
if $q=\infty).$
\end{defin}

\begin{prop}
For $0<\theta <1, \hskip 1truemm 1\leq q <\infty $ or $0\leq \theta \leq 1, \hskip 1truemm q=\infty,$ the functional $||\cdot||_{\theta, q, \overline {X}}$ is a norm on the real interpolation space $[X_0, X_1]_{\theta, q}.$ Endowed with this norm, $[X_0, X_1]_{\theta, q}$ is a Banach space.
\end{prop}

%We recall the following density theorem.

%\begin{thm}\cite[Theorem 2.9]{BS}\label{densityth}
%Let $(X_0, X_1)$ be a compatible couple of Banach spaces and suppose $0<\theta <1, \quad 1\leq q <\infty.$ Then the subspace $X_0 \cap X_1$ is dense in $[X_0, X_1]_{\theta, q}.$
%\end{thm}

\begin{defin}
Let $(X_0, X_1)$ and $(Y_0, Y_1)$ be two compatible couples  of Banach spaces and let $T$ be a linear operator defined on $\overline {X} := X_0 + X_1 $ and taking values in $\overline {Y} := Y_0 + Y_1.$ Then $T$ is said to be admissible with respect to the couples $\overline {X}$ and $\overline {Y}$ if for $i=0, 1,$ the restriction of $T$ is a bounded operator from $X_i$ to $Y_i.$
\end{defin}

\begin{thm}
Let $(X_0, X_1)$ and $(Y_0, Y_1)$ be two compatible couples  of Banach spaces and let $0<\theta<1, \hskip 2truemm 1\leq q <\infty$ or \hskip 1truemm $0\leq \theta \leq 1, \hskip 2truemm q=\infty.$ Let $T$ be an admissible linear operator with respect to the couples $\overline {X}$ and $\overline {Y}$ such that
$$||Tf_i||_{Y_i} \leq M_i||f_i||_{X_i} \quad (f_i \in X_i, \hskip 2truemm i=0, 1).$$
Then $T$ is a bounded operator from $[X_0, X_1]_{\theta, q}$ to $[Y_0, Y_1]_{\theta, q}.$ More precisely, we have
$$||Tf||_{[Y_0, Y_1]_{\theta, q}} \leq M_0^{1-\theta}M_1^{\theta}||f||_{[X_0, X_1]_{\theta, q}}$$
for all $f\in [X_0, X_1]_{\theta, q}.$
\end{thm}

The following theorem gives the real interpolation spaces between Lebesgue spaces and Lorentz spaces on the measure space $(E , \mu).$
\begin{thm}\label{(2.19)}
Let $0<\theta<1, \hskip 2truemm 1\leq q \leq\infty.$ Let $1\leq p_0 < p_1 \leq \infty$ and define the exponent $p$ by $\frac 1p = \frac {1-\theta}{p_0}+\frac \theta {p_1}.$  We have the identifications with equivalence of norms:
\begin{enumerate}
\item[a)]
$ [L_\nu^{p_0}, L_\nu^{p_1}]_{\theta, q}  = L_\nu (p, q);$
\item[b)]
$[L_\nu (p_0, q_0), L_\nu (p_1, q_1)]_{\theta, q} = L_\nu (p, q)$ for
$1<p_0 <p_1< \infty,\hskip 2truemm 1\leq q_0, q_1 \leq \infty.$
\end{enumerate}
\end{thm}

\begin{defin}
Let $(R, \mu)$ and $(S, \nu)$ be two non-atomic $\sigma$-finite measure spaces. Suppose $1\leq p < \infty, \hskip 1truemm 1\leq q \leq \infty.$ Let $T$ be a linear operator defined on the simple functions on $(R, \mu)$ and taking values on the measurable functions on $(S, \nu).$ Then $T$ is said to be of restricted weak type $(p, q)$ if there is a positive constant $M$ such that
$$t^{\frac 1q}(T\chi_F)^\star (t) \leq M\mu (F)^{\frac 1p} \quad (t>0)$$
for all measurable subsets $F$ of $R.$ This estimate can also be written in the form
$$||T\chi_F||_{q, \infty} \leq M||\chi_F||_{p, 1}$$
or equivalently in view of Lemma 2.5,
$$\sup \limits_{\lambda >0} \lambda \mu_{\chi_F} (\lambda)^{\frac 1q} \leq M\mu (F)^{\frac 1p}.$$ 
\end{defin}

In the next two statements, $L_R (p, 1)$ and $L_S (q, \infty)$ denote the corresponding Lorentz spaces on the respective measure spaces $(R, \mu)$ and $(S, \nu).$

\begin{prop}
Let $(R, \mu)$ and $(S, \nu)$ be two non-atomic $\sigma$-finite measure spaces. Suppose $1\leq p < \infty, \hskip 1truemm 1\leq q \leq \infty.$ Let $T$ be a linear operator defined on the simple functions on $(R, \mu)$ and taking values on the measurable functions on $(S, \nu).$ We suppose  that $T$ is of restricted weak type $(p, q).$ Then $T$ uniquely extends to a bounded operator from $L_R (p, 1)$ to $L_S (q, \infty).$ 
\end{prop}

\begin{thm}[Stein-Weiss]
%Let $0<\theta<1, \hskip 2truemm 1\leq q \leq\infty.$
Let $(R, \mu)$ and $(S, \nu)$ be two measure spaces. Suppose $1\leq p_0 < p_1 <  \infty$ and $1\leq q_0, q_1 \leq \infty$ with $q_0 \neq q_1.$ Suppose further that $T$ is a linear operator defined on the simple functions on $(R, \mu)$ and taking values on the measurable functions on $(S, \nu)$ and suppose that $T$ is of restricted weak types $(p_0, q_0)$ and $(p_1, q_1).$ If $1\leq r \leq \infty,$ then $T$ has a unique extension to a linear operator, again denoted by $T,$ which is bounded from $L_R(p, r)$ into $L_S (q, r)$ where
$$\frac 1p = \frac {1-\theta}{p_0}+\frac \theta {p_1}, \quad \frac 1q = \frac {1-\theta}{q_0}+\frac \theta {q_1}, \quad 0<\theta <1.$$
If in addition, the inequalities $p_j \leq q_j \hskip 2truemm (j=0, 1)$ hold, then $T$ is of strong type $(p, q),$ i.e. there exists a positive constant $C$ such that
$$||Tf||_{L^q (S, \nu)} \leq C||f||_{L^p (R, \mu)} \quad (f\in L^p (R, \mu)).$$
\end{thm}

We finish this section with the Wolff reiteration theorem. Let $(X, ||\cdot||_X)$ and $(Y, ||\cdot||_Y)$ be two normed vector spaces over the field $\mathbb C.$ We say that $X$ continuously embeds in $Y$ and we write $X\hookrightarrow Y$ if $X\subset Y$ and there exists a positive constant $C$ such that
$$||x||_Y \leq C||x||_X \quad \quad \forall x\in X.$$
It is easy to check that $X$ identifies with $Y$ if and only if $X\hookrightarrow Y$ and $Y\hookrightarrow X.$\\

\begin{defin}
If $(X_0, X_1)$ is a compatible couple of Banach spaces, then a Banach space $X$ is said to be an intermediate space between $X_0$ and $X_1$ if $X$ is continuously embedded between $X_0 \cap X_1$ and $X_0 + X_1,$ i.e.
$$X_0 \cap X_1 \hookrightarrow X \hookrightarrow X_0 + X_1.$$
\end{defin}

We remind the reader that the real interpolation space $[X_0, X_1]_{\theta, q}, \hskip 2truemm 0<\theta <1, \hskip 2truemm 1\leq q \leq \infty$ is an intermediate space between $X_0$ and $X_1.$ In this direction, we recall the following density theorem.

\begin{thm}\cite[Theorem 2.9]{BS}\label{densityth}
Let $(X_0, X_1)$ be a compatible couple of Banach spaces and suppose $0<\theta <1, \quad 1\leq q <\infty.$ Then the subspace $X_0 \cap X_1$ is dense in $[X_0, X_1]_{\theta, q}.$
\end{thm}

We next state the Wolff reiteration theorem.

\begin{thm}  [\label{wolff} \cite{W, JNP}]
Let $X_2$ and $X_3$ be intermediate Banach spaces of a compatible couple $(X_1, X_4).$ Let $0<\varphi, \psi < 1$ and $1\leq q, r \leq \infty$ and suppose that
$$X_2 = [X_1, X_3]_{\varphi, q}, \quad X_3 = [X_2, X_4]_{\psi, r}.$$
Then (up to equivalence of norms)
$$X_2 = [X_1, X_4]_{\rho, q}, \quad X_3 = [X_1, X_4]_{\theta, r}$$
where
$$\rho = \frac {\varphi \psi}{1-\varphi + \varphi \psi}, \quad \theta = \frac {\psi}{1-\varphi + \varphi \psi}.$$
\end{thm}

\section{The Bergman-Lorentz spaces on tube domains over symmetric cones}

\subsection{Symmetric cones: definitions and preliminary notions}
Materials of this section are essentially from \cite{FK}. We
give some definitions and useful results.

Let $\Omega$ be an irreducible open cone of rank r inside a vector
space $V$ of dimension n, endowed with an inner product $(.|.)$
for which $\Omega$ is self-dual. Such a cone is called a symmetric cone in $V.$
%It is well-known that
%$\Omega$ induces in $V\equiv \bR^n$ a structure of Euclidean
%Jordan algebra, in which $\overline {\Omega}=\{x^2:x\in V\}$. We
%denote by $e$ the identity element in $V$ and by $(x/y)=tr(xy)$
%the canonical inner product.
Let $G(\Omega)$ be the group of transformations of $\Omega$, and
$G$ its identity component. It is well-known that there exists a
subgroup $H$ of $G$ acting simply transitively on $\Omega$, that
is every $y\in \Omega$ can be written uniquely as $y=g\mathbf e$
for some $g\in H$ and a fixed $\mathbf e\in \Omega$. The notation $\Delta$ is for the determinant of $\Omega.$

%Next, we recall the definition of the generalized gamma function on
%$\Omega$:
%$$\Gamma_{\Omega}({\alpha})=\int_{\Omega}e^{-(\mathbf
%{e}|\xi)}\Delta^{\alpha-\frac nr}(\xi)d\xi \quad \quad (\alpha\in
%\mathbb R).$$ This integral converges if and only if $\alpha
%>\frac nr -1.$ Being in this case it is equal to:
%$$\Gamma_{\Omega}(\alpha)=(2\pi)^{\frac{n-r}{2}}\prod_{j=1}^{r}\Gamma\left (\alpha-\frac{\frac nr - 1}{2}\frac {j-1}{r-1}\right )$$
%(see Chapter VII of \cite{FK}).

%\begin{lemma} Let $\alpha
%>\frac nr -1.$ Then for all $y\in \Omega$
%we have
%$$\int_{\Omega}e^{-(y|\xi)}\Delta^{\alpha-\frac nr}(\xi)d\xi=\Gamma_{\Omega}(\alpha)\Delta^{-\alpha}(y).$$
%\end{lemma}

We first recall the following lemma.

\begin{lemma}\cite{BBGNPR}
The following inequality is valid.
$$\Delta (y) \leq \Delta (y+v) \quad \quad (y, v\in \Omega).$$
\end{lemma}

%\begin{defin}
%For $\alpha  >\frac nr -1,$  We define a holomorphic extension to $T_\Omega$ of the function $\Delta^{-\alpha}(y), \quad y\in \Omega,$ by
%$$\Delta^{-\alpha}(\frac zi) := \int_{\Omega}e^{-(\frac zi|\xi)}\Delta^{\alpha-\frac nr}(\xi)d\xi.$$
%\end{defin}

We denote by $d_\Omega$ the $H$-invariant distance on $\Omega.$ The following  lemma will be useful.

\begin{lemma}{\cite{BBGNPR}}\label{invariant}
Let $\delta >0.$ There exists a positive constant $\gamma = \gamma (\delta, \Omega)$ such that for $\xi, \xi' \in \Omega$ satisfying $d_\Omega (\xi, \xi')\leq \delta,$ we have
$$\frac 1\gamma \leq \frac {\Delta (\xi)}{\Delta (\xi')}\leq \gamma.$$
\end{lemma}

%\begin{lemma}
%The measure $\Delta^{-\frac nr}(y)dxdy$ is $H$-invariant on $\Omega.$
%\end{lemma}

In the sequel, we write as usual $V=\mathbb R^n.$

\subsection{Bergman-Lorentz spaces on tube domains over symmetric cones. Proof of Theorem \ref{th1}}
Let $\Omega$ be an irreducible symmetric cone in $\mathbb R^n$ with rank $r,$  determinant $\Delta$ and fixed point $\mathbf e.$ We denote $T_\Omega = \mathbb R^n +i\Omega$ the tube domain in $\mathbb C^n$ over $\Omega.$ For $\nu \in \mathbb R,$ we define the weighted measure $\mu$ on $T_\Omega$ by $d\mu (x+iy) = \Delta^{\nu-\frac nr} (y)dxdy.$ For a measurable subset $A$ of $T_\Omega,$ we denote by $|A|$ the (unweighted) Lebesgue measure of $A,$ i.e. $|A|=\int_A dxdy.$

\begin{defin}
Since the determinant $\Delta$ is a polynomial  in $\mathbb R^n,$ it can be extended in a natural way to $\mathbb C^n$ as a holomorphic polynomial we shall denote $\Delta\left (\frac {x+iy}i \right ).$ It is known that this extension is zero free on the simply connected region $T_\Omega$ in $\mathbb C^n.$ So for each real number $\alpha,$ the power function $\Delta^\alpha$ can also be extended as a holomorphic function $\Delta^\alpha \left (\frac {x+iy}i \right )$ on $T_\Omega.$

%For $\alpha  >\frac nr -1,$  we define a holomorphic extension to $T_\Omega$ of the function $\Delta^{-\alpha}(y), \quad y\in \Omega,$ by
%$$\Delta^{-\alpha}\left (\frac zi \right ) := \int_{\Omega}e^{-(\frac zi|\xi)}\Delta^{\alpha-\frac nr}(\xi)d\xi.$$

\end{defin}

The following lemma will be useful.

\begin{lemma}\cite{BBGNPR}
Let $\alpha  >0.$ 
\begin{enumerate}
\item We have $$\left \vert \Delta^{-\alpha}\left (\frac zi\right )\right \vert \leq \Delta^{-\alpha}(\Im m \hskip 1truemm z) \quad \quad (z\in T_\Omega).$$
\item We suppose $\nu >\frac nr -1.$ The following estimate 
$$\int_\Omega \left (\int_{\mathbb R^n} \left \vert\Delta^{-\alpha}\left (\frac {x+i(y+\mathbf e)}i\right )\right \vert^p dx\right )\Delta^{\nu-\frac nr} (y)dy<\infty$$
holds if and only if $\alpha > \frac {\nu+\frac {2n}r-1}p.$
\end{enumerate}
\end{lemma}

We denote by $d$ the Bergman distance on $T_\Omega.$ We remind the reader that the group $\mathbb R^n \times H$ acts simply transitively on $T_\Omega.$ The following lemma will also be useful.

\begin{lemma}
The measure $\Delta^{-\frac {2n}r}(y)dxdy$ is $\mathbb R^n \times H$-invariant on $T_\Omega.$
\end{lemma}

The following corollary is an easy consequence of Lemma \ref{invariant}.

\begin{cor}\label{Bergman distance}
Let $\delta >0.$ There exists a positive constant $C = C (\delta)$ such that for $z, z' \in \Omega$ satisfying $d (z, z')\leq \delta,$ we have
$$\frac 1C \leq \frac {\Delta (\Im m \hskip 1truemm z)}{\Delta (\Im m \hskip 1truemm z')}\leq C.$$
\end{cor}

The next proposition will lead us to the definition of Bergman-Lorentz spaces.

\begin{prop}
The measure space $(T_\Omega, \mu)$ is a non-atomic $\sigma$-finite measure space.
\end{prop}

\noindent
In view of this proposition, all the results of the previous section are valid on the measure space $(T_\Omega, \mu).$ We shall denote by $L_\nu (p, q)$ the corresponding Lorentz space on $(T_\Omega, \mu).$ Moreover, we write $L_\nu^p$ for the weighted Lebesgue space $L^p (T_\Omega, d\mu)$ on  $(T_\Omega, \mu).$\\
 The following corollary is an immediate consequence of Theorem \ref{densityth} and assertion a) of Theorem 2.17.

\begin{cor}\label{Schwartz}
The subspace $\mathcal C^\infty_c (T_\Omega)$ consisting of $\mathcal C^\infty$ functions with compact support on $T_\Omega$ is dense in the Lorentz space $L_\nu (p, q)$ for all $1<p<\infty, \quad 1\leq q <\infty.$
\end{cor}

\begin{defin}
The Bergman-Lorentz space $A_\nu (p, q), \hskip 2truemm 1\leq p, q\leq \infty$ is the subspace of the Lorentz space $L_\nu (p, q)$ consisting of holomorphic functions. In particular $A_\nu (p, p)=A_\nu^p,$ where $A_\nu^p = Hol (T_\Omega)\cap L_\nu^p$ is the usual weighted Bergman space on $T_\Omega.$ In fact, $A_\nu^p, \hskip 2truemm 1\leq p \leq \infty$ is a closed subspace of the Banach space $L_\nu^p.$ The Bergman projector $P_\nu$ is the orthogonal projector from the Hilbert space $L_\nu^2$ to its closed subspace $A_\nu^2.$
\end{defin}

\begin{exa}
Let $\nu >\frac nr -1, \hskip 2truemm 1 < p < \infty, \hskip 2truemm 1\leq q \leq \infty.$ The function $F(z) = \Delta^{-\alpha} (\frac {z+i\mathbf e}i)$ belongs to the Bergman-Lorentz $A_\nu (p, q)$ if $\alpha > \frac {\nu+\frac {2n}r-1}p.$ Indeed we can find positive numbers $p_0$ and $p_1$ such that $1\leq p_0 < p < p_1 \leq \infty$ and $\alpha > \frac {\nu+\frac {2n}r-1}{p_i} \quad (i=0, 1).$ By assertion 2) of Lemma 3.4, the holomorphic function $F$ belongs to $L_\nu^{p_i} \quad (i=0, 1).$ The conclusion follows by assertion a) of Theorem 2.17 and Theorem 2.22.
\end{exa}

\begin{lemma}
Let $\nu \in \mathbb R, \hskip 2truemm 1< p\leq \infty, \hskip 2truemm 1\leq q \leq \infty$ and let $f\in A_\nu (p, q).$ For every compact set $K$ of $\mathbb C^n$ contained in $T_\Omega,$ there is a positive constant $C_K$ such that
$$|f(z)|\leq C_K ||f||_{p, q} \quad \quad (z\in K).$$
\end{lemma}

\begin{proof}
Suppose first $p=\infty.$ The interesting case is $q=\infty.$ In this case, $L_\nu (p, q)=L_\nu^\infty$ and $A_\nu (p, q) = A_\nu^\infty$ is the space of bounded holomorphic functions on $(T_\Omega, \mu).$ The relevant result is straightforward.\\
We next suppose $1< p< \infty, \hskip 2truemm 1< q <\infty$ and  $f\in A_\nu (p, q).$ Since $L_\nu (p, q)$ continuously embeds in $L_\nu (p, \infty)$ (Lemma 2.7) it suffices to show that for each $f\in A_\nu (p, \infty),$ we have 
$$|f(z)|\leq C_K ||f||_{p, \infty} \quad \quad (z\in K).$$
For a compact set $K$ in $\mathbb C^n$ contained in $T_\Omega,$ we call $\rho$ the Euclidean distance from $K$ to the boundary of $T_\Omega.$ We denote $B(z, \frac \rho 2)$ the Euclidean ball centered at $z,$ with radius $\frac \rho 2.$ We apply successively
\begin{itemize}
\item[-]
the mean-value property,
\item [-] the fact that the function
$$u+iv\in T_\Omega \mapsto \Delta^{\nu-\frac nr} (v)$$
is uniformly bounded below on every Euclidean ball $B(z, \frac \rho 2)$ when $z$ lies on $K$ and
\item [-] the second part of Theorem \ref{HL},
\end{itemize}
to obtain that 
$$
\begin{array}{clcr}
|f(z)|&=\frac 1{|B(z, \frac \rho 2)|}\left \vert \int_{B(z, \frac \rho 2)} f(u+iv)dudv\right \vert\\
&\leq \frac C{|B(z, \frac \rho 2)|}\int_{B(z, \frac \rho 2)} |f(u+iv)|\Delta^{\nu-\frac nr} (v)dudv\\
&\leq \frac C{|B(z, \frac \rho 2)|}\int_0^{\mu (B(z, \frac \rho 2))} f^\star (t)dt\\
&\leq \frac C{|B(z, \frac \rho 2)|}\int_0^{\mu (K_\rho)} t^{\frac 1p}f^\star (t)t^{\frac 1{p'}}\frac {dt}t\\
&\leq \frac {C ||f||_{p, \infty}}{|B(z, \frac \rho 2|}\int_0^{\mu (K_\rho)} t^{\frac 1{p'}}\frac {dt}t \leq C_K ||f||_{p, \infty}
%&\leq \frac C{|B(z, \frac \rho 2)|}(\int_0^{\mu (K_\rho)} (t^{\frac 1p}f^\star (t))^q \frac {dt}t)^{\frac 1q}(\int_0^{\mu (K_\rho)} t^{\frac {q'}{p'}}\frac {dt}t)^{\frac 1{q'}}\\
%&=C_K \int_0^\infty (t^{\frac 1p}f^\star (t))^q \frac {dt}t)^{\frac 1q}
\end{array}
$$
for each $z\in K$ with $K_\rho = \bigcup_{z\in K} \overline {B}(z, \frac \rho 2)$ and $C_K = Cp'\frac 1{|B(z, \frac \rho 2)|}(\mu (K_\rho))^{\frac 1{p'}}.$ %At the latter but one line, we applied the H\"older inequality. 
We recall that there is a positive constant $C_n$ such that for all $z\in \mathbb C^n$ and $\rho >0,$ we have $|B(z, \rho)|=C_n \rho^{2n}$ and we check easily that $\mu (K_\rho)<\infty.$\\
%Now suppose $1< p< \infty, \hskip 2truemm q =\infty.$ Then 
%$$
%\begin{array}{clcr}
%|f(z)&\leq \frac C{|B(z, \frac \rho 2)|}\int_0^{\mu (K_\rho)} t^{\frac 1p}f^\star (t)t^{\frac 1{p'}}\frac {dt}t\\
%&\leq C'_K \sup \limits_{t>0} t^{\frac 1p}f^\star (t)
%\end{array}
%$$
%with $C'_K = \frac {Cp'}{|B(z, \frac \rho 2)|}(\mu (K_\rho))^{\frac 1{p'}}.$\\
%Finally the case where $1< p< \infty, \hskip 2truemm q =1$ can be treated in the same way.
\end{proof}

%\begin{thm}
%Let $1< p\leq \infty$ and $1\leq q \leq \infty.$ Equipped with the norm $||\cdot||_{p, q},$ the Bergman-Lorentz space $A_\nu (p, q)$ is a Banach space.
%\end{thm}

\begin{proof}[\textbf {Proof of Theorem \ref{th1}}]
(1) We suppose that $\nu \leq \frac nr -1.$ It suffices to show that $A_\nu (p, \infty) =\{0\}$ for all $1< p< \infty.$ Given $F\in A_\nu (p, \infty),$ we first prove the following lemma.

\begin{lemma}
For general $\nu \in \mathbb R,$ the following estimate holds.
\begin{equation}\label{1}
|F(x+i(y+\mathbf e))|\Delta^{\frac {\nu+\frac nr}{p}} (y+\mathbf e) \leq C_p ||F||_{p, \infty} \quad \quad (x+iy\in T_\Omega).
\end{equation}
\end{lemma}

\begin{proof} [Proof of the Lemma]
We recall the following inequality \cite{BBGNPR}:
\begin{equation*}
|F(x+i(y+\mathbf e))|\leq C\int_{d(x+i(y+\mathbf e), \hskip 1truemm u+iv)<1} |F(u+iv)|\frac {dudv}{\Delta^{\frac {2n}r} (v)}
\end{equation*}
\begin{equation}\label{2} 
\leq C' \Delta^{-\nu-\frac nr} (y+\mathbf e)\int_{d(x+i(y+\mathbf e), \hskip 1truemm u+iv)<1} |F(u+iv)|d\mu (u+iv).
\end{equation}
The latter inequality follows by Corollary \ref{Bergman distance}. Now by Theorem 2.2, we have
\begin{equation*}
\int_{d(x+i(y+\mathbf e), u+iv)<1} |F(u+iv)|d\mu (u+iv)
\leq \int_0^{\mu \left (B_{berg}(x+i(y+\mathbf e), 1)\right )} t^{\frac 1p}F^\star (t)t^{\frac 1{p'}}\frac {dt}t
\end{equation*}
\begin{equation}\label{3}
\leq p'||F||_{p, \infty}\left (\mu (B_{berg}(x+i(y+\mathbf e), 1)\right )^{\frac 1{p'}},
\end{equation}
where $B_{berg} (\cdot, \cdot)$ denotes the Bergman ball in $T_\Omega.$ By Lemma 3.5 and Corollary 3.6, we obtain that 
\begin{equation}\label{4}
\mu \left (B_{berg}(x+i(y+\mathbf e), 1)\right ) \simeq \Delta^{\nu+\frac nr} (y+\mathbf e).
\end{equation}
Then combining  (\ref{2}), (\ref{3}) and (\ref{4}) gives the announced estimate (\ref{1}).
\end{proof}

We next deduce that the function 
$$z\in T_\Omega \mapsto F(z+i\mathbf e)\Delta^{-\alpha} (z+i\mathbf e)$$
belongs to the Bergman space $A^1_\nu$ when $\alpha$ is sufficiently large. We distinguish two cases: 1) $\nu \leq -\frac nr;$ 2) $-\frac nr < \nu \leq \frac nr -1.$\\
{\underline {Case}} 1). We suppose that $\nu \leq -\frac nr.$ We take $\alpha >\frac {-\nu -\frac nr}p$ and we apply assertion (1) of Lemma 3.4  to get
$$
\left \vert F(x+i(y+\mathbf e))\Delta^{-\alpha} (x+i(y+\mathbf e))\right \vert$$
%\begin{array}{clcr}
$$=\left \vert F(x+i(y+\mathbf e))\right \vert \left \vert \Delta^{-\alpha -\frac {\nu +\frac nr}p} (x+i(y+\mathbf e))\right \vert \left \vert \Delta^{\frac {\nu +\frac nr}p} (x+i(y+\mathbf e))\right \vert$$
$$\leq\left \vert F(x+i(y+\mathbf e))\right \vert \left \vert \Delta^{-\alpha -\frac {\nu +\frac nr}p} (x+i(y+\mathbf e))\right \vert \left \vert\Delta^{\frac {\nu +\frac nr}p} (y+\mathbf e)\right \vert$$
$$\leq C_p ||F||_{p, \infty}|\Delta^{-\alpha -\frac {\nu +\frac nr}p} (x+i(y+\mathbf e))|.
%\end{array}
$$
For the latter inequality, we applied estimate (\ref{1}) of Lemma 3.12. The conclusion follows because by  assertion (2) of Lemma 3.4, the function $\Delta^{-\alpha -\frac {\nu +\frac nr}p} (x+i(y+\mathbf e))$ is integrable on $T_\Omega$ when $\alpha$ is sufficiently large.\\
{\underline {Case}} 2). We suppose that $-\frac nr < \nu \leq \frac nr -1.$ Since $\nu +\frac nr >0$ and $\Delta (y+\mathbf e) \geq \Delta (y)$ by Lemma 3.1, it follows from (\ref{1}) that the function $z\in T_\Omega \mapsto F(z+i\mathbf e)$ is bounded on $T_\Omega.$ The conclusion easily follows.

Finally we remind that $A^1_\nu =\{0\}$ if $\nu \leq \frac nr -1$ (cf. e.g. \cite{BBGNPR}). We conclude that the function $F(\cdot +i\mathbf e)$ vanishes identically on $T_\Omega.$ An application of the analytic continuation principle then implies the identity $F\equiv 0$ on $T_\Omega.$
\vskip 2truemm
(2) We suppose that $\nu >\frac nr -1.$ It suffices to show that $A_\nu (p, q)$ is a closed subspace of the Banach space $(L_\nu (p, q), ||\cdot||_{p, q}).$ For $p=\infty,$ the interesting case is $q=\infty$ and then $A_\nu (p, q) = A_\nu^\infty;$ the relevant result is easy to obtain. \\
We next suppose that $1< p< \infty$ and $1\leq q \leq \infty.$ In view of Lemma 3.11, every Cauchy sequence $\{f_m\}_{m=1}^\infty$ in $\left (A_\nu (p, q), ||\cdot||_{p, q}\right )$ converges to a holomorphic function $f:T_\Omega \rightarrow \mathbb C$ on compact sets in $\mathbb C^n$ contained in $T_\Omega.$ On the other hand, since the sequence  $\{f_m\}_{m=1}^\infty$  is a Cauchy sequence in the Banach space $\left (L_\nu (p, q), ||\cdot||_{p, q}\right ),$ it converges with respect to the $L_\nu (p, q)$-norm to a function $g\in L_\nu (p, q).$ Now by Proposition 2.10, this sequence contains a subsequence $\{f_{m_k}\}_{k=1}^\infty$ which converges $\mu$-a.e. to $g.$ The uniqueness of the limit implies that $f=g$ a.e. We have proved that the Cauchy sequence $\{f_m\}$
in $(A_\nu (p, q), ||\cdot||_{p, q})$ converges in $(A_\nu (p, q), ||\cdot||_{p, q})$ to the function $f.$
\end{proof}

\section{Density in Bergman-Lorentz spaces. Proof of Theorem \ref{th2}}
%\label{3.11}
\subsection{Density in Bergman-Lorentz spaces.}
We adopt the following notation given in the introduction:
$$Q_\nu = 1+\frac \nu{\frac nr -1}.$$
We shall refer to the following  result. For its proof, consult \cite{S} and \cite{BBGNPR}.

\begin{thm}\label{thm}
Let $\nu >\frac nr -1.$ The weighted Bergman projector $P_\gamma, \hskip 1truemm \gamma \geq \nu + \frac \nu r -1$ (resp. the Bergman projector $P_\nu)$ extends to a bounded  operator from $L^p_\nu$ to $A_\nu^p$ for all $1\leq p < Q_\nu$ (resp. for all $1+Q_\nu^{-1} <  p < 1+ Q_\nu).$
\end{thm}

The following corollary follows from a combination of Theorem \ref{thm}, Theorem 2.16 and Theorem 2.17.

\begin{cor}
Let $\nu >\frac nr -1.$ The weighted Bergman projector $P_\gamma, \hskip 1truemm \gamma \geq \nu + \frac \nu r -1$ (resp. the Bergman projector $P_\nu)$ extends to a bounded  operator from $L_\nu (p, q)$ to $A_\nu (p, q)$ for all  $1 <p < Q_\nu$ (resp. for all $1+Q_\nu^{-1} <  p < 1+ Q_\nu)$ and $1\leq q\leq \infty.$
\end{cor}
The following proposition was proved in \cite{BBGNPR}.

\begin{prop}
We suppose that $\nu, \gamma > \frac nr -1.$ Let $1\leq p< \infty.$  The subspace $A_\nu^p \cap A^t_\gamma$ is dense in  the Banach space $A_\nu^p.$ Moreover, if the weighted Bergman projector $P_\gamma$ extends to a bounded operator on $L^p_\nu$ and if $P_\gamma$ is the identity on $A^t_\gamma,$ then $P_\gamma$ is the identity on $A^p_\nu.$
\end{prop}

The next corollary is a consequence of Theorem 4.1 and Proposition 4.3.

\begin{cor}\label{4.3}
Let $\gamma > \frac nr -1.$ For all $t\in (1+Q_\gamma^{-1}, 1+Q_\gamma),$ the Bergman projector $P_\gamma$ is the identity on $A^t_\gamma.$
\end{cor}

We shall prove the following density result for Bergman-Lorentz spaces.

\begin{prop}\label{density}
We suppose that $\gamma \geq \nu > \frac nr -1, \hskip 2truemm 1< p, t< \infty$ and $1\leq q < \infty.$ The subspace $A_\nu (p, q)\cap A^t_\gamma$ is dense in  the Banach space $A_\nu (p, q)$ in the following three cases.
\begin{enumerate}
\item $p=q;$
\item $\gamma=\nu >\frac nr -1, \hskip 2truemm 1+Q_\nu^{-1}<p, t <1+Q_\nu$ and $1\leq q <p;$
\item $p, q\in (t, \infty).$
\end{enumerate}
\end{prop}

\begin{proof}
%In view of Lemma 2.10, we shall show the proposition for the quasi-norm $||\cdot||^\star_{p, q}$ on $A_\nu (p, q).$

$\it {1)}$ For $p=q,\hskip 2truemm A_\nu (p, p)=A_\nu^p,$ the result is known, cf. e.g. \cite{BBGNPR}.
\vskip 1truemm
$\it {2)}$  We suppose now that $\gamma=\nu >\frac nr -1, \hskip 2truemm 1+Q_\nu^{-1}<p, t <1+Q_\nu$ and $1\leq q <p.$ Given $F\in A_\nu (p, q),$ by Corollary \ref{Schwartz}, there exists a sequence $\{f_m\}_{m=1}^\infty$ of 
$\mathcal C^\infty$ functions with compact support on $T_\Omega$ such that $\{f_m\}_{m=1}^\infty \rightarrow F \hskip 2truemm (m\rightarrow \infty)$ in $L_\nu (p, q).$ Each $f_m$ belongs to $L_\nu^t \cap L_\nu (p, q).$ By Corollary 4.2 and Theorem 4.1, the Bergman projector $P_\nu$ extends to a bounded operator on $L_\nu (p, q)$ and on $L_\nu^t$ respectively. So $P_\nu f_m \in A_\nu (p, q)\cap A^t_\nu$ and $\{P_\nu f_m\}_{m=1}^\infty \rightarrow P_\nu F \hskip 2truemm (m\rightarrow \infty)$ in $A_\nu (p, q).$ Notice that $A_\nu (p, q) \subset A_\nu^p$ because $q<p.$ By Corollary 4.4, we obtain that $P_\nu F=F.$ This finishes the proof of assertion $\it 2).$

%, the Bergman projector $P_\nu$ extends to a bounded operator on $L_\nu (p, q)$ for $1+Q_\nu^{-1}<p <1+Q_\nu$ and $1\leq q <p.$ Moreover by Theorem 4.1 and Corollary 4.4, $P_\nu$ is bounded on $L_\nu^t$ and $P_\nu$ is the identity on $A_\nu^t.$ So $P_\nu f_m \in A_\nu (p, q)\cap A^t_\nu$ and $\{P_\nu f_m\}_{m=1}^\infty \rightarrow P_\gamma F \hskip 2truemm (m\rightarrow \infty)$ in $A_\nu (p, q).$ Notice that $A_\nu (p, q) \subset A_\nu^p$ because $q<p.$ By  Corollary 4.4, we obtain $P_\nu F=F.$ This finishes the proof of assertion $\it 2).$
\vskip 1truemm
$\it {3)}$ Let $l$ be a bounded linear functional on $A_\nu (p, q)$ such that $l(F) =0$ for all $F\in A_\nu (p, q)\cap A^t_\gamma.$ We must show that $l\equiv 0$ on $A_\nu (p, q).$\\
We first prove that the holomorphic function $F_{m, \alpha}$ defined on $T_\Omega$ by
$$F_{m, \alpha} (z) = \Delta^{-\alpha} \left (\frac {\frac zm +i\mathbf e}i\right )F (z)$$
belongs to $A_\nu (p, q)\cap A^t_\gamma$ when $\alpha >0$ is sufficiently large. By  Lemma 3.4 (assertion (1)) and Lemma 3.1, we have
$$\left \vert \Delta^{-\alpha} \left (\frac {\frac zm +i\mathbf e}i \right ) \right \vert \leq \Delta^{-\alpha}  (\mathbf e)=1.$$
So $\left \vert F_{m, \alpha}\right \vert \leq |F|;$ this implies that $F_{m, \alpha} \in A_\nu (p, q).$\\
We next show that $F_{m, \alpha} \in A_\gamma^t$ when $\alpha$ is large. We obtain that
$$
\begin{array}{clcr}
I&:=\int_{T_\Omega} {\left \vert \Delta^{-\alpha} \left (\frac {\frac {x+iy}m +i\mathbf e}i \right )F(x+iy) \right \vert}^{t} \Delta^{\gamma -\frac nr} (y)dxdy\\
&=\int_{T_\Omega} {\left \vert \Delta^{-\alpha} \left (\frac {\frac {x+iy}m +i\mathbf e}i \right )F(x+iy) \right \vert}^{t} \Delta^{\gamma -\nu} (y)d\mu (x+iy).
\end{array}
$$  
We take $\alpha t > \gamma - \nu.$ By  Lemma 3.4 (assertion (1)) and Lemma 3.1, we obtain that
$$I\leq \int_{T_\Omega} {\left \vert \Delta^{-\alpha t +\gamma -\nu} \left (\frac {\frac {x+iy}m +i\mathbf e}i \right ) \right \vert}|F(x+iy)|^t d\mu (x+iy).$$
Observing that $(|f|^t)^\star = (f^\star)^t,$ we notice  that $|F|^t \in L_\nu (\frac pt, \frac qt).$ By Theorem \ref{duality}, it suffices to show that the function $z\in T_\Omega \mapsto \Delta^{-\alpha t +\gamma -\nu} \left (\frac {\frac {x+iy}m +i\mathbf e}i \right )$ belongs to $L_\nu ({(\frac pt)', (\frac qt)'})$ when $\alpha$ is large. The desired conclusion follows by Example 3.10.\\
So our assumption implies that
\begin{equation}\label{dual1}
l(F_{m, \alpha})=0.
\end{equation}
By the Hahn-Banach theorem, there exists a bounded linear functional $\widetilde l$ on $L_\nu (p, q)$ such that $\widetilde l\vert_{L_\nu (p, q)}=l$ and $||\widetilde l||=||l||.$ Furthermore, by Theorem \ref{duality}, there exists a function $\varphi \in L_\nu (p', q')$ such that
$$\widetilde l(f) = \int_{T_\Omega} f(z)\varphi (z)d\mu (z) \quad \quad \forall f\in L_\nu (p, q).$$
We must show that
\begin{equation}\label{dual2}
\int_{T_\Omega} F(z)\varphi (z)d\mu (z)=0 \quad \quad \forall F\in A_\nu (p, q).
\end{equation}
The equation (\ref{dual1}) can be expressed in the form
$$\int_{T_\Omega} F_{m, \alpha}(z)\varphi (z)d\mu (z)=\int_{T_\Omega} \Delta^{-\alpha} \left (\frac {\frac zm +i\mathbf e}i\right )F(z)\varphi (z)d\mu (z)=0.$$
Again by Theorem \ref{duality}, the function $F\varphi$ is integrable on $T_\Omega$ since $F\in L_\nu (p, q)$ and $\varphi \in L_\nu (p', q').$ We also have 
$$\left \vert \Delta^{-\alpha} \left (\frac {\frac zm +i\mathbf e}i\right )\right \vert \leq 1 \quad \quad \forall z\in T_\Omega, \hskip 2truemm m=1,2,\cdots$$
An application of the Lebesgue dominated theorem next gives the announced conclusion (\ref{dual2}).
\end{proof}

We next deduce the following corollary. 

\begin{cor}\label{identity}
Let $\nu >\frac nr -1.$ 
\begin{enumerate}
\item
For all $1< p <Q_\nu$ and $1< q <\infty,$ there exists a real index $\gamma \geq\nu+\frac nr -1$ such that the Bergman projector $P_\gamma$ is the identity on $A_\nu (p, q).$
\item
For all $p\in (1+Q_\nu^{-1}, 1+Q_\nu)$ and $1\leq q <\infty,$ the Bergman projector $P_\nu$ is the identity on $A_\nu (p, q).$
\end{enumerate}

%The Bergman projector $P_\gamma, \hskip 2truemm \gamma \geq\nu +\frac nr -1$ (resp. $P_\nu$) is the identity on $A_\nu (p, q)$ for all $1\leq q< p <Q_\nu$ (resp. $1+Q_\nu^{-1}<p <1+Q_\nu$ and $1\leq q \leq p).$
\end{cor}

\begin{proof}
{\it (1)} By Corollary 4.4, for all $t\in (1+Q_\gamma^{-1}, 1+Q_\gamma),$ $P_\gamma$ is the identity on $A^t_\gamma.$ Next let $1< p <Q_\nu$ and $1< q <\infty.$ Let the real index $\gamma$ be such that $\gamma \geq\nu+\frac nr -1$ and $1+Q_\gamma^{-1} < \min (p, q).$ We take $t$ such that $1+Q_\gamma^{-1} <t < \min (p, q).$ It follows from the assertion {(\it 3)} of Proposition 4.5 that the subspace $A_\nu (p, q)\cap A_\gamma^t$ is dense in the Banach space $A_\nu (p, q).$ But by Corollary 4.2, $P_\gamma$ extends to a bounded operator on $L_\nu (p, q).$ We conclude then that $P_\gamma$ is the identity on $A_\nu (p, q).$\\
\indent
{\it (2)} By Corollary 4.4, for all $t\in (1+Q_\nu^{-1}, 1+Q_\nu),$ $P_\nu$ is the identity on $A^t_\nu.$ Next let $1+Q_\nu^{-1}< p <1+Q_\nu.$
By Corollary 4.2, $P_\nu$ extends to a bounded operator on $L_\nu (p, q).$ It then suffices to show that the subspace $A_\nu (p, q)\cap A_\nu^t$ is dense in the Banach space $A_\nu (p, q)$ for some $t\in (1+Q_\nu^{-1}, 1+Q_\nu).$ If $1+Q_\nu^{-1} <q <\infty,$ we take $t$ such that $1+Q_\nu^{-1} <t < \min (p, q):$ the conclusion follows by assertion $(3)$ of Proposition 4.5. Otherwise, if $1\leq q \leq 1+Q_\nu^{-1},$ the conclusion follows by assertion $(2)$ of the same proposition.  

  %We distinguish the cases $1+Q_\gamma^{-1} <q <\infty$ and $1<q \leq 1+Q_\gamma^{-1}$ for which we apply assertions (3) and (2) of Proposition 4.4 respectively. 
%Let the real index $\gamma$ be such that $\gamma \geq\nu+\frac nr -1$ and $1+Q_\gamma^{-1} < \min (p, q).$ 
%We take $t$ such that $1+Q_\nu^{-1} <t < \min (p, q).$ It follows from Proposition 4.4 that the subspace $A_\nu (p, q)\cap A_\nu^t$ is dense in the Banach space $A_\nu (p, q).$ But by Corollary 4.2, $P_\nu$ extends to a bounded operator on $L_\nu (p, q).$ We conclude then that $P_\nu$ is the identity on $A_\nu (p, q).$

%For $P_\gamma,$ take $\gamma = \gamma$ in the proof of Proposition 4.4.
\end{proof}

\subsection{Proof of Theorem 1.2.} Combine Corollary 4.2 with Corollary \ref{identity}.

\section{Real interpolation between Bergman-Lorentz spaces. Proof of Theorem \ref{th3}.} 
%We shall first prove  assertion $(1)$ of Theorem \ref{th3}. 
For all $1< p_0 < p_1 < \infty$ and $1\leq q_0, \hskip 1truemm q_1 \leq \infty$ and $0<\theta<1,$ by Defintion 2.13 (the definition of real interpolation spaces), we have at once  
$$[A_\nu (p_0, q_0), A_\nu (p_1, q_1)]_{\theta, q}\hookrightarrow [L_\nu (p_0, q_0), L_\nu (p_1, q_1)]_{\theta, q} = L_\nu (p, q),$$
with $\frac 1p = \frac {1-\theta}{p_0}+\frac \theta {p_1}.$
We conclude that $[A_\nu (p_0, q_0), A_\nu (p_1, q_1)]_{\theta, q}\hookrightarrow Hol (T_\Omega) \cap  L_\nu (p, q) = A_\nu (p, q).$\\
We next show the converse, i.e. $A_\nu (p, q) \hookrightarrow [A_\nu (p_0, q_0), A_\nu (p_1, q_1)]_{\theta, q}.$ If we denote $A_0 = A_\nu (p_0, q_0), \hskip 1truemm A_1 = A_\nu (p_1, q_1), \hskip 1truemm \overline {A}= A_0 +A_1,$ we must show that there exists a positive constant $C$ such that
\begin{equation}\label{star1}
||F||_{[A_0, A_1]_{\theta, q}} \leq C||F||_{L_\nu (p, q)}\quad \quad \forall F\in A_\nu (p, q).
\end{equation}
We recall that
\begin{equation}\label{star2}
||F||_{[A_0, A_1]_{\theta, q}} = \left (\int_0^\infty \left (t^{-\theta} K(t, F, \overline {A}) \right )^q \frac {dt}t \right )^{\frac 1q}
\end{equation}
with
$$K(t, F, \overline {A}) = \inf \hskip 2truemm \{||a_0||_{A_0} + t||a_1||_{A_1}: F=a_0 +a_1, a_0 \in A_0, a_1\in A_1\}.$$
We denote $L_0 = L_\nu (p_0, q_0), \hskip 1truemm L_1 = L_\nu (p_1, q_1), \hskip 1truemm \overline {L}= L_0 +L_1.$ Since $L_\nu (p, q)$ identifies with $[L_\nu (p_0, q_0), L_\nu (p_1, q_1)]_{\theta, q}$ with equivalence of norms (Theorem 2.17), we have
\begin{equation}\label{star3}
||F||_{L_\nu (p, q)} = \left (\int_0^\infty \left (t^{-\theta} K(t, F, \overline {L}) \right )^q \frac {dt}t \right )^{\frac 1q}.
\end{equation}
Comparing (\ref{star2}) and (\ref{star3}), the estimate (\ref{star1}) will follow if we can prove that
$$K(t, F, \overline {A})\leq CK(t, F, \overline {L})\quad \quad \forall F\in A_\nu (p, q).$$
\indent
We first suppose that $1< p_0 < p_1 <Q_\nu, \hskip 1truemm 1\leq q_0, q_1\leq \infty$  and  $0<\theta<1.$  Let $\frac 1p=\frac {1-\theta}{p_0}+\frac \theta {p_1}, \hskip 1truemm 1< q<\infty$ and let $F\in A_\nu (p, q).$ By Theorem 1.2,  for $\gamma$ sufficiently large we have $F=P_\gamma F$ and by Corollary 4.2, $P_\gamma$ extends to bounded operator from $L_i$ to $A_i, \hskip 2truemm i=0, 1.$ So the inclusion
$$\{(a_0, a_1)\in A_0 \times A_1: F=a_0 + a_1\}$$
$$\supset \{(P_\gamma a_0, P_\gamma a_1): (a_0, a_1)\in L_0 \times L_1, F=a_0 + a_1\}$$
implies that
$$
K(t, F, \overline {A})$$
$$\leq \inf \hskip 1truemm \left \{||P_\gamma a_0||_{L_0} + t||P_\gamma a_1||_{L_1}: F=a_0 +a_1, a_0 \in L_0, a_1\in L_1 \right \}$$
$$\leq \inf \hskip 1truemm \left \{||P_\gamma||_0 ||a_0||_{L_0} + t||P_\gamma||_1 ||a_1||_{L_1}: F=a_0 +a_1, a_0 \in L_0, a_1\in L_1 \right \}
$$
where $||\cdot||_i$ denotes the operator norm on $L_i \quad (i=0, 1).$ We have shown that
$$K(t, F, \overline {A})\leq \left (\max \limits_{i=0, 1} ||P_\gamma||_i \right )K(t, F, \overline {L}).$$
Remind that $\max \limits_{i=0, 1} ||P_\gamma||_i <\infty.$ This completes the proof in this first case.\\
\indent
We next suppose that $1+Q_\nu^{-1}< p_0 < p_1 <1+Q_\nu , \hskip 1truemm 1\leq q_0, q_1\leq \infty$  and  $0<\theta<1.$ Let $\frac 1p=\frac {1-\theta}{p_0}+\frac \theta {p_1}, \hskip 1truemm 1\leq q<\infty$ and let $F\in A_\nu (p, q).$   Again by Theorem 1.2, we have $F=P_\nu F$  and by Corollary 4.2, $P_\nu$ extends to bounded operator from $L_i$ to $A_i, \hskip 2truemm i=0, 1.$ So the inclusion
$$\{(a_0, a_1)\in A_0 \times A_1: F=a_0 + a_1\}\supset \{(P_\nu a_0, P_\nu a_1): (a_0, a_1)\in L_0 \times L_1, F=a_0 + a_1\}$$
implies that
$$
\begin{array}{clcr}
K(t, F, \overline {A})&\leq \inf \hskip 1truemm \left \{||P_\nu a_0||_{L_0} + t||P_\nu a_1||_{L_1}: F=a_0 +a_1, a_0 \in L_0, a_1\in L_1 \right \}\\
&\leq \inf \hskip 1truemm \left \{||P_\nu||_0 ||a_0||_{L_0} + t||P_\nu||_1 ||a_1||_{L_1}: F=a_0 +a_1, a_0 \in L_0, a_1\in L_1 \right \}\\
&\leq \left (\max \limits_{i=0, 1} ||P_\nu ||_i \right )K(t, F, \overline {L}).
\end{array}
$$
This finishes the proof in the second case.\\

%Let $\gamma \geq \nu +\frac nr -1.$ By Theorem \ref{th2} and Theorem 2.18, if we denote $L_0 = L_\nu (p_0, q_0), \hskip 1truemm L_1 = L_\nu (p_1, q_1), \hskip 1truemm \bar L= L_\nu (p_0, q_0)+L_\nu (p_1, q_1),$ we have
%$$
%\begin{array}{clcr}
%A_\nu (p, q) &= P_\gamma L_\nu (p, q)\\
%& = P_\gamma ([L_0, L_1]_{\theta, q})\\
%&=\{P_\gamma a: a\in \overline {L}, ||a||_{\theta, q, \overline {L}} < \infty\}.
%\end{array}
%$$ 
%We recall that
%$$||a||_{\theta, q, \overline {L}} = (\int_0^\infty t^{-\theta} K(t, a, \overline {L})^q \frac {dt}t)^{\frac 1q}$$
%if $1\leq q < \infty$ (resp.
%$$||a||_{\theta, \infty, \overline {L}} = \sup \limits_{t>0} t^{-\theta} K(t, a, \overline {L})$$
%if $q=\infty)$ where
%$$K(t, a, \overline {L}) = \inf \hskip 2truemm \{||a_0||_{L_0} + t||a_1||_{L_1}: a=a_0 +a_1, a_0 \in L_0, a_1\in L_1\}.$$
%If we denote $A_0 = A_\nu (p_0, q_0), \hskip 1truemm A_1 = A_\nu (p_1, q_1), \hskip 1truemm \bar A= A_0 +A_1,$ 
% we obtain
%$$
%K(t, P_\gamma a, \overline {A})$$
%$$\leq \inf \hskip 2truemm \{||P_\gamma a_0||_{L_0} + t||P_\gamma a_1||_{L_1}: a=a_0 +a_1, a_0 \in L_0, a_1\in L_1\}$$
%$$\leq \inf \hskip 2truemm \{||P_\gamma||_0 ||a_0||_{L_0} + t||P_\gamma ||_1 ||a_1||_{L_1}: a=a_0 +a_1, a_0 \in L_0, a_1\in L_1\}
%$$
%where $||\cdot||_i$ denote the operator norm on $L_i \quad (i=0, 1).$ So
%$$K(t, P_\gamma a, \overline {A}) \leq \left (\max \limits_{i=0, 1} ||P_\gamma||_i\right ) \hskip 1truemm K(t, a, \overline {L}).$$
%By Theorem 4.1, we have $\max \limits_{i=0, 1} ||P_\gamma||_i < \infty.$ This completes the proof.\\
\indent
%The case where $1+Q_\nu^{-1} <  p_0 <p_1 < 1+ Q_\nu$ is treated in the same manner, replacing $P_\gamma$ by $P_\nu.$ \\

We now prove the second assertion of Theorem \ref{th3}. We combine the Wolff reiteration theorem (Theorem \ref{wolff}) with the first assertion of the theorem. Here $X_1 = A_\nu (p_0, q_0)$ and $X_4 = A_\nu (p_1, q_1).$ We take $X_2 = A_\nu (p_2, r)$ and $X_3 = A_\nu (p_3, q)$  with $1+Q_\nu^{-1} < p_2 < p_3 < Q_\nu, 1\leq q< \infty$ and $1< r < \infty.$ By assertion (1) of the theorem, we have
$$X_2 = [X_1, X_3]_{\varphi, q} \quad {\rm with} \quad \frac 1{p_2} = \frac {1-\varphi}{p_0} + \frac \varphi {p_3}$$ 
and
$$X_3 = [X_2, X_4]_{\psi, r} \quad {\rm with} \quad \frac 1{p_3} = \frac {1-\psi}{p_2} + \frac \psi {p_1}.$$
 By Theorem \ref{wolff}, we conclude that
$$X_2 = [X_1, X_4]_{\rho, q}, \quad X_3 = [X_1, X_4]_{\theta, r}$$
with
$$\rho = \frac {\varphi \psi}{1-\varphi + \varphi \psi}, \quad \theta = \frac {\psi}{1-\varphi + \varphi \psi}.$$
In other words, in the present context, we have
$$A_\nu (p_2, q) = [A_\nu (p_0, q_0), A_\nu (p_1, q_1)]_{\rho, q}, \hskip 1truemm A_\nu (p_3, r) = [A_\nu (p_0, q_0), A_\nu (p_1, q_1)]_{\theta, r}$$
with
$$\rho = \frac {\varphi \psi}{1-\varphi + \varphi \psi}, \quad \theta = \frac {\psi}{1-\varphi + \varphi \psi}.$$
To conclude, given $1+Q_\nu^{-1} < p< Q_\nu,$ take for instance $p_3 =p.$

\section{Open questions.}
\subsection{Statement of Question 1}
In \cite{BBGRS}, the following conjecture was stated.\\
{\textbf {Conjecture.}} Let $\nu >\frac nr -1.$ Then the Bergman projector $P_\nu$ admits a bounded extension to $L^p_\nu$ if and only if
$$p'_\nu <  p<  p_\nu := \frac {\nu +\frac {2n}r -1}{\frac nr -1} - \frac {(1-\nu)_+}{\frac nr -1}. $$

This conjecture was completely settled recently for the case of tube domains over Lorentz cones ($r=2$). The proof of this result is a combination of results of \cite{BBGR} and \cite{BD} (cf.  \cite{BGN}; cf. also \cite{BoNa} for the particular case where $\nu=\frac n2$). In this case, Theorem \ref{th3} can be extended as follows.

\begin{thm}
We restrict to the particular case of tube domains over Lorentz cones {\rm (}r=2{\rm)}. Let $\nu >\frac n2 -1.$ 
\begin{enumerate}
%\item
%For all $1\leq p_0 < p_1 < 1+Q_\nu,$ the real interpolation space $[A_\nu^{p_0}, A_\nu^{p_1}]_{\theta, q}$ identifies with $A_\nu (p, q)$ with equivalence of norms.
\item
For all $1<  p_0 < p_1 < Q_\nu$ (resp. $p'_\nu <  p_0 <p_1 < p_\nu$)    and $0<\theta<1,$ the real interpolation space $[A_\nu (p_0, q_0), A_\nu (p_1, q_1)]_{\theta, q}$ identifies with $A_\nu (p, q), \hskip 1truemm \frac 1p = \frac {1-\theta}{p_0}+\frac \theta {p_1}, \hskip 2truemm 1< q <\infty$ (resp. $1\leq q <\infty)$ with equivalence of norms.
\item 
For  $1<  p_0 < p'_\nu < p_1 < p_\nu$ and $1\leq q_0, q_1 \leq \infty,$ the Bergman-Lorentz spaces $A_\nu (p, q), \hskip 2truemm p'_\nu < p < Q_\nu, \hskip 2truemm 1< q <\infty$ are real interpolation spaces between $A_\nu (p_0, q_0)$ and $A_\nu (p_1, q_1).$
\end{enumerate}
\end{thm}

In the upper rank case ($r\geq 3$), we always have $(1-\nu)_+=0$ and the conjecture has the form
$$p'_\nu <  p<  p_\nu := \frac {\nu +\frac {2n}r -1}{\frac nr -1}.$$
The best result so far is Theorem \ref{th2}. 
If we could show that $P_\nu$ is of restricted weak type $(p_1, p_1)$ for some $1+Q_\nu < p_1 < p_\nu$ (resp. $p'_\nu < p_1 <1+Q_\nu^{-1}),$ then by the Stein-Weiss interpolation theorem (Theorem 2.20), we would obtain that $P_\nu$ admits a bounded extension to $L_\nu^p$ for all $1+Q_\nu \leq p <p_1$ (resp. $p'_\nu < p_1 <1+Q_\nu^{-1}).$ This would improve Theorem \ref{th2}.
\vskip 2truemm
\noindent
{\textbf{Question 1.}} Prove the existence of an exponent $p_1 \in (1+Q_\nu, p_\nu) $ (resp. $p_1 \in (p'_\nu, 1+Q_\nu^{-1}))$ such that $P_\nu$ is of restricted weak type $(p_1, p_1).$ That is, there exists a positive constant $C_{p_1}$ such that for each measurable subset $E$ of $T_\Omega,$ we have
$$\sup \limits_{\lambda >0} \lambda^{p_1} \mu_{\chi_E} (\lambda)\leq C_{p_1}\mu (E).$$

For Lorentz cones ($r=2$), the latter property holds for all  $p_1 \in (p'_\nu, 1+Q_\nu^{-1})\cup (1+Q_\nu, p_\nu).$

\subsection{Question 2} 
The second question is twofold; it is induced by Proposition 4.5 and Corollary 4.6. Let $\nu >\frac nr -1.$
\begin{enumerate}
\item 
Here $1+Q_\nu^{-1}\leq q<\infty.$ Is $A_\nu (p, q)\cap A_\nu^t$ dense in $A_\nu (p, q)$ for some (all) $1+Q_\nu^{-1}<p\leq t< 1+Q_\nu$ and $p\in (q, \infty)?$
\item
Here $q=1.$ Suppose that $p \in (1, 1+Q_\nu^{-1}].$ Is $A_\nu (p, 1)\cap A_\gamma^t$ dense in $A_\nu (p, 1)$ for some (all) $ t\in (1, \infty)$ and $\gamma >\frac nr -1?$
\end{enumerate}

\subsection{Question 3}
%The third question is induced by Corollary 4.6.
Can the second assertion of Theorem 1.3 be extended to some values $p \in (1, 1+Q_\nu^{-1}]$ or $p\in [Q_\nu, 1+Q_\nu)?$

%\section{An unsolved question.}
%Prove the density of the subspace $A_\nu (p, q) \cap A_\gamma^2$ in the space $A_\nu (p, q)$ for all $\nu, \gamma > \frac nr -1$ and $1\leq p, q \leq \infty.$ This is known for $q=p$ and was proved in Proposition 4.4 for $q<  p$ in three particular cases. The case $q>p$ is completely open even for the upper half-plane.

\vskip 5truemm
\noindent
\textbf{Acknowledgements.} The authors wish to express their gratitude to Aline Bonami, Jacques Faraut, Gustavo Garrig\'os and Chokri Yacoub for valuable discussions.

\bibliographystyle{plain}

\end{document}